\documentclass{amsart}
\usepackage{verbatim}
\usepackage{rotating} 
\usepackage{amssymb}
\usepackage{mathrsfs,amsfonts,amsmath,amssymb,epsfig,amscd,xy,amsthm,pb-diagram} 

\newtheorem{theorem}{Theorem}[section]
\newtheorem{proposition}[theorem]{Proposition}

\newtheorem{lemma}[theorem]{Lemma}

\newtheorem{corollary}[theorem]{Corollary}
\newtheorem{question}[theorem]{Question}
\newtheorem{definition}[theorem]{Definition}

\theoremstyle{plain}

\theoremstyle{remark}

\newtheorem{remark}[theorem]{Remark}

\newtheorem{assume}[theorem]{Assumption}

\newcommand{\tm}{\tilde{m}}
\newcommand{\tn}{\tilde{n}}
\newcommand{\C}{{\mathbb C}}

\newcommand{\Q}{{\mathbb Q}}

\newcommand{\R}{{\mathbb R}}
\newcommand{\Z}{{\mathbb Z}}
\newcommand{\N}{{\mathbb N}}

\newcommand{\cC}{{\mathcal C}}
\newcommand{\cA}{{\mathcal A}}

\newcommand{\fq}{\mathfrak q}
\newcommand{\fr}{\mathfrak r}

\newcommand{\Qbar}{\bar{\Q}}

\DeclareMathOperator{\Gal}{Gal}

\DeclareMathOperator{\Norm}{N}

\DeclareMathOperator{\GL}{GL}

\DeclareMathOperator{\id}{id}

\newcommand{\bP}{{\mathbb P}}

\newcommand{\bfx}{{\mathbf x}}

\newcommand{\cO}{\mathcal{O}}
\newcommand{\cB}{\mathcal{B}}

\DeclareMathOperator{\disc}{disc}

\author{Khoa D.~Nguyen}
\address{
Khoa D.~Nguyen \\
Department of Mathematics\\
University of British Columbia\\
And Pacific Institute for The Mathematical Sciences\\ 
Vancouver, BC V6T 1Z2, Canada}
\email{dknguyen@math.ubc.ca}
\urladdr{www.math.ubc.ca/\~{}dknguyen}


\begin{document}
	\title{On Modules of Integral Elements Over Finitely Generated Domains}
	\date{Apr 20, 2015}
	\begin{abstract}		 
		This paper is motivated by the results and questions of Jason P.~Bell and Kevin G.~Hare in \cite{BellHare}. Let $\cO$ be a finitely generated $\Z$-algebra
		that is an integrally closed domain of characteristic zero. 
		We investigate  
    the following two problems:
		\begin{itemize}
			\item [(A)] Fix $q$ and $r$ that are integral over $\cO$,
			describe all pairs $(m,n)\in\N^2$ such that $\cO[q^m]=\cO[r^n]$.  
			\item [(B)] Fix $r$ that is integral over $\cO$, describe all $q$
			such that $\cO[q]=\cO[r]$.
		\end{itemize}	
		 In this paper, we solve Problem (A), present a solution of Problem (B) by
		 Evertse and Gy\H{o}ry, and explain their relation to the paper
		 of Bell and Hare.
		 In the following, $c_1$ and $c_2$
    are effectively computable constants with
    a very mild dependence on $\cO$, $q$, and $r$. For (B), Evertse and Gy\H{o}ry 
     show that there are $N\leq c_2$ elements $s_1,\ldots,s_N$
    such that $\cO[s_i]=\cO[r]$ for every $i$, and for every $q$ such that
    $\cO[q]=\cO[r]$, we have $q-us_i\in\cO$ for some $1\leq i\leq N$
    and $u\in\cO^{*}$. This immediately answers two questions about Pisot
    numbers by Bell and Hare \cite{BellHare}. For (A), we show that except for some 
    ``degenerate'' cases that can be explicitly described, there are at most 
    $c_1$ such pairs $(m,n)$. This significantly strengthens some results
    in \cite{BellHare}. We also make some remarks on effectiveness and discuss
    further questions at the end of the paper.
	\end{abstract}
	
	\maketitle
	
	\section{Introduction} \label{sec:intro}
	Throughout this paper, $\N$ denotes the set of positive integers. For simplicity, 
	the terminology \emph{finitely generated domain} means 
	an integral domain of characteristic 0 finitely generated as an algebra over $\Z$.
	Fix 
	an embedding $\Qbar\subset \C$. A Pisot number is a real algebraic
		integer greater than 1 whose other conjugates are of modulus less than 1. In 
	\cite{BellHare}, an algebraic integer $q$ of degree $d\geq 2$ over
	$\Q$ is of full rank   if the multiplicative group
	of $\C^*$ generated by the conjugates of $q$ either has rank $d$, 
	or has rank $d-1$ and the norm of $q$ is $\pm 1$. 
	The following very interesting results are
	established in \cite{BellHare} (also see \cite{BellHare-Corrigendum}):
	\begin{itemize}
		\item [(i)] Fix an algebraic integer $q$ of full rank and positive integer $n$, the set of $m\in\N$ such that
		$\Z[q^m]=\Z[q^n]$ is finite (\cite[Theorem 1.1]{BellHare}).
		\item [(ii)] Let $q$ and $r$ be full rank
		algebraic integers of degree $d\geq 2$. Then except certain explicit ``degenerate'' cases,
		the set of $m\in\N$ such that $\Z[q^m]=\Z[r^m]$ is finite (\cite[Theorem 1.3]{BellHare}).
		\item [(iii)] Fix an algebraic integer $r$ such that $\Q(r)/\Q$ is Galois, there are
		only finitely many Pisot numbers $q$ such that $\Z[q]=\Z[r]$ (\cite[Theorem 1.6]{BellHare}).
	\end{itemize}
	Bell and Hare also ask two questions involving part (iii): is it possible to give a bound depending on $r$ and to remove 
	the assumption 
	on $\Q(r)/\Q$? 
	
	\emph{From now on,
	$\cO$ denotes
	an integrally closed finitely generated domain with fraction field $K$}. The typical and most important examples are rings of
	integers in number fields. In our main result, we fix $q$ and $r$ that are
	integral over $\cO$ such that \emph{$q^n$ and $r^n$ are not in $\cO$
	for every $n\in\N$}, and study the equation $\cO[q^m]=\cO[r^n]$
	in \emph{both} variables $(m,n)$. Our result significantly strengthen
	the above results in (i) and (ii) of Bell and Hare \cite{BellHare} 
	at one stroke. When $\cO=\Z$ as in \cite{BellHare}, it is
	obvious that the condition of being full rank implies
	the very mild condition $\{q^n,r^n:\ n\in\N\}\cap \cO=\emptyset$.
	This latter condition is assumed in order to simplify
	the statements of our results stated in this section. It comes
	from the minor inconvenience that
	$\cO[t]=\cO$ for every $t\in\cO$ no matter how large
	the ``height'' of $t$ is. We will also explain how our arguments
	could handle the case when some $q^n$ or $r^n$ is in $\cO$ (see Section~\ref{sec:add}), hence provide
	a complete (in a certain \emph{qualitative} sense) solution
	to the problem of describing
	$(m,n)$ such that $\cO[q^m]=\cO[r^n]$ even without
	the above condition on $q$ and $r$. A remarkable feature
	of our result is that it provides a uniform bound with
	a very mild dependence on the data $(\cO,q,r)$ illustrated
	below (see Remark~\ref{rem:Dirichlet}).
	
	A theorem of Roquette \cite{Roquette} (also see \cite[Chapter~2]{LangFundamentals}) states that the group
	of units in a finitely generated domain 
	is finitely generated. Hence $\cO^*$ has only finitely many torsion points.
	In other words, there are only finitely many roots of unity
	in $K$.   
	 We
	need the following:
	\begin{definition}\label{def:unit over}
	Let  $\alpha$ be integral over $\cO$. We say that $\alpha$ is a unit over $\cO$ if
	$\Norm_{K(\alpha)/K}(\alpha)\in \cO^{*}$, where $\Norm_{K(\alpha)/K}$
	is the norm map with respect to $K(\alpha)/K$. By using
	the minimal polynomial of $\alpha$ over $K$, this is equivalent to
	requiring that $\alpha$ is a unit in $\cO[\alpha]$.
	\end{definition}
	
	\begin{definition}\label{def:dAab}
	Let $\alpha$ and $\beta$ be integral over $\cO$. The notation $d(\cO,\alpha,\beta)$
	denotes the maximum of all the following numbers:
	\begin{itemize}
		\item [(a)] $[K(\alpha):K]$ and $[K(\beta):K]$.
		\item [(b)] The rank of the group of units of $\cO[\sigma(\alpha),\sigma(\beta),\tau(\alpha),\tau(\beta)]$
		for any two $K$-embeddings $\sigma$ and $\tau$ of $K(\alpha,\beta)$ into 
		$\bar{K}$.
		\item [(c)] The number of roots of unity in $K$.
	\end{itemize} 
	\end{definition}

  Although the definition of $d(\cO,\alpha,\beta)$	looks somewhat complicated, we have 
  the
  following simple observation:
  \begin{remark}\label{rem:Dirichlet}
   If $K$ is a number field, $S$ is a finite set of places of $K$ containing all the 
   archimedean ones, $\cO$ is the ring of $S$-integers of $K$,  and $\alpha$ and $\beta$
   are integral over $\cO$, then $d(\cO,\alpha,\beta)$ could be
   bounded explicitly in terms of $\max\{[K(\alpha):\Q],[K(\beta):\Q]\}$
   and the cardinality of $S$. This follow from (the $S$-unit version of) 
   Dirichlet's unit theorem.
  \end{remark}
 
	Before stating our main result, we need to define subsets of $\N^2$
	corresponding to certain ``degenerate'' cases (this name comes from degenerate
	solutions of certain unit equations considered later). Let $q$ and $r$ be 
	integral over $\cO$ such that $\{q^n,r^n:\ n\in\N\}\cap \cO=\emptyset$.  Define
	the subsets $\cA_{\cO,q,r}$, $\cB_{\cO,q,r}$, $\cC_{\cO,q,r}$
	as follows:	

	$$\cA_{\cO,q,r}:=\{(m,n)\in\N^2:\ \frac{q^m}{r^n}\in\cO^{*}\},$$
   $$\cB_{\cO,q,r}:=\{(m,n)\in\N^2:\ [K(r^n):K]=2\ \text{and 
   }\frac{q^m}{\sigma(r^n)}\in\cO^{*}\}$$
   where $\sigma$ in the definition of $\cB_{\cO,q,r}$ is
   the nontrivial automorphism of the quadratic extension
   $K(r^n)/K$. Finally, if $q$ and $r$ are units over $\cO$, we define:
   $$\cC_{\cO,q,r}:=\{(m,n)\in\N^2:\ q^mr^n\in\cO^{*}\};$$
   otherwise define $\cC_{\cO,q,r}=\emptyset$. By
   our assumption on $q$ and $r$, we have $\cA_{\cO,q,r}\cap(\cB_{\cO,q,r}\cup\cC_{\cO,q,r})=\emptyset$. On the other hand, when
   $r$ is a unit over $\cO$, we have $\cB_{\cO,q,r}\subseteq\cC_{\cO,q,r}$.

	Obviously, if $(m,n)\in \cA_{\cO,q,r}$ then $\cO[q^m]=\cO[r^n]$. If $(m,n)\in \cB_{\cO,q,r}$, note that $[K(r^n):K]=2$ and let $\sigma$ denote
	the nontrivial $K$-automorphism of $K(r^n)$. Using the fact that $r^n+\sigma(r^n)\in \cO$, we have that $\cO[q^m]=\cO[\sigma(r^n)]=\cO[r^n]$. Finally,
	if $(m,n)\in \cC_{\cO,q,r}$, put $u=q^mr^n\in\cO$. Using the minimal
	polynomials of $r^n$ and $q^m$ over $K$, we have
	$\displaystyle q^m=\frac{u}{r^n}\in \cO[r^n]$
	and $\displaystyle r^n=\frac{u}{q^m}\in \cO[q^m]$ hence $\cO[q^m]=\cO[r^n]$. Because of this, the 
	following result is, in a certain qualitative sense, best possible:
	\begin{theorem}\label{thm:main 2}
		Let $q$ and $r$ be integral over $\cO$ such
		that $\{q^n,r^n:\ n\in\N\}\cap \cO=\emptyset$. There
		is an effectively computable constant $c_3$ depending only on
		$d(\cO,q,r)$
		 such that outside
		$\cA_{\cO,q,r}\cup\cB_{\cO,q,r}\cup\cC_{\cO,q,r}$
		there are at most $c_3$ pairs $(m,n)\in\N^2$
		satisfying $\cO[q^m]=\cO[r^n]$. 
	\end{theorem}
	
	The bound $c_3$ as well as other similar bounds in this paper follow from work
	 of Beukers and Schlickewei \cite{BeuSch},
	and Evertse, Schlickewei, and Schmidt \cite{ESS} on unit
	equations together with some combinatorial arguments. Hence it is fairly
	straightforward to make them explicit although we do not provide all
	the details in doing so. 
	Conceivably, all
	these bounds are far from optimal. In this paper, we do not
	spend any efforts to optimize them as long as
	they depend uniformly only on $d(\cO,q,r)$
	instead of $\cO$, $q$, and $r$.
	Theorem~\ref{thm:main 2} immediately implies the following (see
	Theorem~1.1 and Theorem~1.3 in \cite{BellHare}).
	\begin{corollary}\label{cor:1.1}
		Let $K$ be a number field with the ring of integers $\cO_K$. Let $q$ be an 
		algebraic integer such that $q^n\notin \cO_K$ 
		for every $n\in\N$. There is
		an effectively computable constant $c_4$ depending only on
		$[K(q):\Q]$ such that there are at most $c_4$
		pairs $(m,n)\in\N^2$ satisfying $m\neq n$ and $\cO_K[q^m]=\cO_K[q^{n}]$.
	\end{corollary}
	\begin{proof}
	We apply Theorem~\ref{thm:main 2} with $r=q$ and $\cO=\cO_K$;
	the resulting bound $c_4$ only depends on $[K(q):\Q]$ thanks to
	Remark~\ref{rem:Dirichlet}.
	By the assumption on $q$,
	the sets $\cB_{\cO,q,q}$ and $\cC_{\cO,q,q}$ are empty while the set
	$\cA_{\cO,q,q}$ is exactly the set of pairs $(m,n)$ with $m=n$.
	\end{proof}
	
	\begin{corollary}\label{cor:1.3}
		Let $q$ and $r$ be algebraic integers. There is
		an effectively computable constant $c_5$ 
		depending only on $\max\{[\Q(q):\Q],[\Q(r):\Q]\}$ such that
		the following holds.
		If there are more than $c_5$ numbers $n\in\N$ such that
		$\Z[q^n]=\Z[r^n]$ then we have one of the following:
		\begin{itemize}
			\item [(a)] $\displaystyle\frac{q}{r}$ is a root of unity.
			\item [(b)] $\displaystyle\frac{q}{r_1}$ is a root of unity for
			some conjugate $r_1\neq r$. Moreover, this
			case can only happen if for some $n\in\N$, we have $[\Q(r^n):\Q]=2$
			and $r_1^n\neq r^n$.
			\item [(c)] $qr$ is a root of unity.
		\end{itemize}
	\end{corollary}
	\begin{proof}
		We apply Theorem~\ref{thm:main 2} when $K=\Q$ and $\cO=\Z$;
		the resulting bound $c_5$ only depends on $\max\{[\Q(q):\Q],[\Q(r):\Q]\}$
		by Remark~\ref{rem:Dirichlet}.
		Since there are more than $c_5$ many
		$n$ such that $\Z[q^n]=\Z[r^n]$, by Theorem~\ref{thm:main 2}
	  there is some $n_0$ such that $(n_0,n_0)\in 
	  \cA_{\Z,q,r}\cup\cB_{\Z,q,r}\cup\cC_{\Z,q,r}$. Cases (a), (b), and (c) respectively
	  come from the cases where $(n_0,n_0)$ belong to $\cA_{\Z,q,r}$, $\cB_{\Z,q,r}$
	  and $\cC_{\Z,q,r}$.
	\end{proof}
	
	The previous corollaries strengthen results by Bell and Hare mentioned in parts
	(i) and (ii) at the beginning of this paper. For their questions asked in part 
	(iii), we fix $r$ which is integral over $\cO$ and study the collection
	of $q$ satisfying the equation 
	$\cO[q]=\cO[r]$. Evertse and Gy\H{o}ry \cite{EG1985} prove that
	there exists a positive integer $N$ uniformly bounded by an explicit quantity with 
	a mild dependence
	on $\cO$ and $r$
	such that the following holds.
	There are $s_1,\ldots,s_N$ such that $\cO[s_i]=\cO[r]$ for $1\leq i\leq N$
	and for every $q$ satisfying $\cO[q]=\cO[r]$, we have
	$q-us_i\in\cO$ for some $u\in \cO^{*}$. Furthermore, if 
	$\cO\subseteq\bar{\Q}$ or if $\cO$ belongs to certain
	families of integrally closed finitely generated domains
	then Gy\H{o}ry proves that a list $\{s_1,\ldots,s_N\}$ satisfying the above 
	property can be
	determined effectively. We refer the readers to \cite{Gyory1984},
	\cite{EG1985} and the references there for more details. In 
	Section~\ref{sec:EvertseGyory}, we briefly explain how to prove the above results 
	by
	Gy\H{o}ry and Evertse-Gy\H{o}ry and why they immediately answer
	the questions by Bell and Hare.
	
	To prove Theorem~\ref{thm:main 2}, we start with the equality 
	$\cO[q^m]=\cO[r^n]$
	and its consequence that the discriminants of $q^m$ and $r^n$ over $K$
	differ by an element in $\cO^*$. Then
	it is straightforward to obtain a list 
	of solutions of the unit equation $x+y+z=1$ where $x,y,z$ belong
	to a subgroup of $\bar{K}^{*}$ whose rank is bounded in terms of
	$d(\cO,q,r)$. A celebrated
	result of Evertse, Schlickewei and Schmidt \cite{ESS} provides a uniform bound 
	for the
	number of nondegenerate solutions (i.e. when $x,y,z\neq 1$). On the other hand,
	when many solutions are degenerate, it is
	not obvious to
	get out the exact relations as described in
	the definition of $\cA_{\cO,q,r}$, $\cB_{\cO,q,r}$, and $\cC_{\cO,q,r}$.  
	Some extra combinatorial and Galois theoretic arguments are needed for this 
	remaining 
	problem 
	where we bound the  
	number of degenerate solutions that are outside
	$\cA_{\cO,q,r}\cup\cB_{\cO,q,r}\cup\cC_{\cO,q,r}$. The proof of 
	the results by Gy\H{o}ry and Evertse on the equation $\cO[q]=\cO[r]$
	follows the same 
	idea, but we have a simpler equation of the form $x+y=1$ in variables $x,y$ 
	instead. 
	For this equation,
	there is an earlier result by Beukers and Schlickewei \cite{BeuSch} that provides
	a reasonably good uniform bound on the number of solutions. Moreover, when
	$x$ and $y$ are taken inside a finitely generated subgroups of
	$\bar{\Q}^{*}$, the
	equation $x+y=1$ 
  can be solved effectively using Baker's method or
  the Thue-Siegel principle (see, for example, \cite{GyoryYu} and \cite[pp.~146--148]{BG}).  
	
	In the next section, we present results involving discriminants and, more 
	importantly, the above results on unit equations. After that, we present
	briefly the work of Evertse and Gy\H{o}ry on
	the equation $\cO[q]=\cO[r]$ with a given $r$.
	Then we prove Theorem~\ref{thm:main 2} and explain how to
	remove the very mild condition $\{q^n,r^n:\ n\in\N\}\cap \cO=\emptyset$
	assumed there. This provides a complete solution to
	the problem of describing all $(m,n)$ such that $\cO[q^m]=\cO[r^n]$
	for \emph{any} $q$ and $r$. At the end, we discuss the effectiveness
	of the results in this paper and
	some related questions.

	{\bf Acknowledgments.}  We wish to thank Professors Jason Bell, Mike Bennett,
	Jan-Hendrik Evertse, Dragos Ghioca, K\'alm\'an Gy\H{o}ry, Kevin Hare, and the anonymous referee for many
	helpful suggestions to improve the paper.

	\section{Some Preliminary Results}\label{sec:preliminary}
	\subsection{Discriminants}
	Let $A$ be an integrally closed domain with fraction field $E$ of characteristic 0. For any 
	$\alpha$ algebraic over $E$ of degree $D$  with conjugates 
	$\alpha_0=\alpha,...,\alpha_{D-1}$
	we define the discriminant of  $\alpha$ over $E$ to be:
	$$\disc_E(\alpha):=\prod_{0\leq i<j<D} (\alpha_i-\alpha_j)^2.$$
	
	We have the following well-known result:
	\begin{proposition}\label{prop:discriminant}
	If $\alpha$ and $\beta$ are integral over $A$ satisfying $A[\alpha]=A[\beta]$ then 
	$\disc_E(\alpha)=u\disc_E(\beta)$ for some $u\in A^*$. 
	\end{proposition}
	\begin{proof}
	See \cite[pp.~20--21]{Ribenboim}. Although the results there are stated over number
	fields, the proof can be carried over without any change. Integral closedness of $A$
	is needed for the fact that $A[\alpha]$ (respectively $A[\beta]$) is \emph{free} 
	with basis $1,\alpha,\ldots,\alpha^{D-1}$ 
	(respectively $1,\beta,\ldots,\beta^{D-1}$) where $D=[E(\alpha):E]=[E(\beta):E]$.
%
	\end{proof}
	
	\subsection{Unit equations in 3 variables}
	Fix $n\geq 2$, we start with the following:
	\begin{definition}\label{def:discriminant}
  Let $a_1,\ldots,a_n\in \C^*$. A solution $(u_1,\ldots,u_n)\in (\C^*)^n$ of 
  the equation $a_1x_1+\ldots+a_nx_n=1$
  in variables $x_1,\ldots,x_n$
  is called nondegenerate if no subsums vanish. In other words, there
  is no proper subset $\emptyset\neq J\subset\{1,\ldots,n\}$ such that
  $\sum_{j\in J} a_ju_j=0$.
	\end{definition}
	
	Equations of the form $a_1x_1+\ldots+a_nx_n=1$ where each $x_i$ is
	an $S$-unit in a number field have played a fundamental role in
	diophantine geometry since work of Siegel in the 1920s. After many decades
	of intense activities, Evertse, Schlickewei and Schmidt obtained the
	following celebrated result with a remarkable uniform bound 
	\cite[Theorem~1.1]{ESS}:
	\begin{theorem}\label{thm:ESS}
	Suppose $\Gamma$ is a subgroup of $(\C^*)^n$ of rank $R$. Consider
	the equation:
	$$a_1x_1+\ldots+a_nx_n=1$$
	in variables $x_1,\ldots,x_n$. Then the number
	of nondegenerate solutions in $\Gamma$ 
	is at most $\exp((6n)^{3n}(R+1))$.
	\end{theorem}
	
	As a consequence, we have the following:
	\begin{corollary}\label{cor:3 variables}
	 Let $G$ be a subgroup of $\C^{*}$ of rank $R$, there are at most
	 $\exp(18^{9}(3R+1))$ nondegenerate solutions
	 $(x_1,x_2,x_3)\in G^3$ of the equation $x_1+x_2+x_3=1$.
	\end{corollary}
	
	\subsection{Unit equations in 2 variables}
	The special case of Theorem~\ref{thm:ESS} when $n=2$ 
	was obtained 
	earlier by Beukers and Schlickewei \cite[Theorem~1.1]{BeuSch}. It has
	the immediate consequence:
	\begin{corollary}\label{cor:2 variables}
		Let $G$ be a subgroup of $\C^{*}$ of rank $R$,
		there are at most $2^{16R+16}$
		solutions $(x,y)\in G^2$ of the equation
		$ax+by=1$. 
	\end{corollary}
	\begin{proof}
	Let $\Gamma=G\times G$ which has rank $R^2$. Beukers and Schlickewei \cite[Theorem~1.1]{BeuSch} consider the equation $X+Y=1$. We can transform
	the given equation $ax+by=1$ into that equation by enlarging $\Gamma$
	with $(a,1)$ and
	$(1,b)$.	
	\end{proof}
	
	Let $h:\ \Qbar\rightarrow \R_{\geq 0}$ be the absolute logarithmic Weil height
	(see \cite[Chapter~1]{BG}). By using Baker's theory of linear forms in logarithms
	\cite{Baker1975}, \cite{BakerWustholz93}, \cite{Yu2007}
	or the Thue-Siegel principle \cite{Bombieri93}, \cite{BomCoh97},
	we can solve 
	the equation $ax+by=1$ effectively when $a,b\in\Qbar^{*}$ and
	the variables $x$ and $y$ take values in a finitely generated
	subgroups of $\Qbar^{*}$. The readers are referred to \cite[Theorem 1]{GyoryYu}, 
	\cite[pp.~146--148]{BG}, and the references there for more details.
	We have:
	\begin{theorem}\label{thm:effective}
		Let $a,b\in\Qbar^{*}$ and $G$ be a finitely generated subgroup of $\Qbar^{*}$.
		There is an effectively computable constant $c_6(a,b,G)$
	  depending on $a$, $b$, and $G$ such that every solution $(x,y)\in G^2$
	  of $ax+by=1$ satisfies $\max\{h(x),h(y)\}\leq c_6(a,b,G)$.
	\end{theorem}
	
	\begin{remark}\label{rem:EGunit}
		Recently, Evertse and Gy\H{o}ry \cite{EGunit2013} showed that 
		Theorem~\ref{thm:effective} still holds without the condition that 
		$a,b\in \bar{\Q}^*$ and $G\subset\bar{\Q}^*$. In their results,
		we need to express
		the finitely generated domain
		$\Z[a,b,g_1,\ldots,g_R]$
		where $g_1,\ldots,g_R$ are generators of $G$
		into the form $\Z[x_1,\ldots,x_m]/I$ and replace
		the height function on $\bar{Q}$ by a certain ``size'' function. We do not
		use this result here and refer the readers to \cite{EGunit2013}.
	\end{remark}
	
	\section{Results of Evertse-Gy\H{o}ry and questions of Bell-Hare}\label{sec:EvertseGyory}
	For the rest of this section,  fix
	an integrally closed finitely generated domain $\cO$ with fraction field $K$. We need the 
	following:
	\begin{definition}\label{def:dAa}
	Let $\alpha$
	be integral over $\cO$. The notation $d(\cO,\alpha)$ denotes
	the maximum of the following
	\begin{itemize}
		\item [(a)] $[K(\alpha):K]$
		\item [(b)] The rank of the group of units of $\cO[\sigma(\alpha),\tau(\alpha)]$
		for any two $K$-embeddings $\sigma$ and $\tau$ of $K(\alpha)$ into $\bar{K}$.
	\end{itemize} 
	\end{definition}
	
	As before, we have the following:
	\begin{remark}\label{rem:Dirichlet 2}
	If $K$ is  a number field, $\cO$ is the ring of $S$-integers in $K$, and 
	$\alpha$ is integral over $\cO$ then $d(\cO,\alpha)$ could be bounded
	explicitly in terms of $[K(\alpha):\Q]$ 
	and the cardinality of $S$.
	\end{remark}

	The following result is established by Gy\H{o}ry \cite{Gyory1984} and Evertse-Gy\H{o}ry \cite{EG1985}:
	\begin{theorem}[Evertse-Gy\H{o}ry]\label{thm:EG}
	Let $r$ be integral over $\cO$. There is an effectively computable
	constant $c_7$ depending only on $d(\cO,r)$ such that the following holds.
	There are $N\leq c_7$ numbers $s_1,\ldots,s_N$
  such that $\cO[s_i]=\cO[r]$ for $1\leq i\leq N$, and for every
  $q$ satisfying $\cO[q]=\cO[r]$, we have $q-us_i\in\cO$ for some
  $1\leq i\leq N$ and some $u\in\cO^*$. Moreover, when $\cO\subset \bar{\Q}$, such
  a list of $s_1,\ldots,s_N$ can be determined effectively. 
	\end{theorem}
	
		After a series of work, Gy\H{o}ry \cite{Gyory1984} proved that
		there was such a \emph{finite} list $\{s_1,\ldots,s_N\}$
		and it could be determined effectively when $\cO\subset\bar{\Q}$ 
		or $\cO$ belonged to certain restricted classes of finitely generated domains. The assertion that $c_7$
		could be explicitly given and depended only on $d(\cO,r)$ was proved 
		later by Evertse and Gy\H{o}ry \cite{EG1985}. Strictly speaking,
		they represented $\cO$ and $\cO^*$
		after choosing a transcendence basis and a finite set of valuations.
		 Then they worked on the general theory of decomposable form equations
		 and obtained a variant of Theorem~\ref{thm:EG} as an immediate consequence. In 
		 this section,
		 we briefly explain the very simple aspects of their work by using the unit 
		 equation 
		 $x+y=1$
		 (or actually $ax+by=1$ with parameters $(a,b)$)
		 directly to obtain Theorem~\ref{thm:EG}. We will avoid all the extra technical 
		 details
		  for the more general decomposable form equations; the interested readers
		  can refer to \cite{Gyory1984}, \cite{EG1985} and the references there.

	Before proving Theorem~\ref{thm:EG}, we note the following immediate corollary 
	which answers the two questions of Bell and Hare 
	mentioned in (iii) at the beginning of this paper:
	\begin{corollary}\label{cor:2 questions}
	Fix an algebraic integer $r\notin \Z$. The
	number of Pisot numbers $q$ satisfying $\Z[q]=\Z[r]$ could be 
	bounded uniformly in the degree of $r$. Moreover, all
	such Pisot numbers can be determined effectively.
	\end{corollary}
	\begin{proof}
		Apply Theorem~\ref{thm:EG} with $\cO=\Z$ and $K=\Q$, we obtain 
		$c_8$ depending only on $[\Q(r):\Q]$ (see Remark~\ref{rem:Dirichlet 2})
		such that there are $N\leq c_8$ algebraic integers $s_1,\ldots,s_N$
		satisfying the conclusion of Theorem~\ref{thm:EG}. In particular, 
		we have $q=us_i+k$ for some
		$1\leq i\leq N$, $u\in\{\pm 1\}$, and $k\in\Z$. For a fixed $i$ and
		$u$, there are at most two choices of $k$ since we can pick a nontrivial
		embedding $\sigma$ of $\Q(r)$ into $\bar{\Q}$ and use the fact that
		$|\sigma(q)|=|\sigma(us)+k|<1$. Hence there are at most $4c_8$ such Pisot numbers
		$q$. Since a list $\{s_1,\ldots,s_N\}$ can be determined effectively, so can 
		the collection of all such Pisot numbers.
	\end{proof}
		
 \begin{remark}\label{rem:nothing}		
	There is nothing special
	about being a Pisot number in (the proof of) Corollary~\ref{cor:2 questions}. The same
 arguments could be used for any collection of numbers $q$ satisfying some
 	 appropriate 
	boundedness condition that could be much weaker than conditions in the definition
	of Pisot numbers. For instance, we may consider the collection of $q$ such that
	there is a nontrivial embedding $\sigma$ of $\Q(r)$ satisfying
	the condition that $|\sigma(q)|$
	is bounded above by a constant.
	\end{remark}

	We now spend the rest of this section to prove Theorem~\ref{thm:EG}. We may assume
	that $r\notin K$, otherwise Theorem~\ref{thm:EG} is obvious. Let $L/K$ be the Galois 
	closure of
	$K(r)/K$. For any two distinct $K$-embeddings $\sigma$ and $\eta$
	of $K$ into $L$, write $\cO_{\sigma,\eta}=\cO[\sigma(r),\eta(r)]$. For 
	simplicity, 
	write $d=d(\cO,r)$ and let $c_9(d),c_{10}(d),\ldots$
	denote positive constants depending only on $d$. Let $q$ be integral over $\cO$ 
	such that $\cO[q]=\cO[r]$. We have that 
	for every two distinct $K$-embeddings $\sigma$ and $\eta$ of $K(r)$ into $L$, there is
	a unit $u_{\sigma,\eta}$ of $\cO_{\sigma,\eta}^*$ such that:
	\begin{equation}\label{eq:unit 1}
	\sigma(q)-\eta(q)=u_{\sigma,\eta}(\sigma(r)-\eta(r)).
	\end{equation}
	
	\textbf{Case 1:} $[K(r):K]=2$. We can uniquely write $q=a_0+a_1r$
	with $a_0,a_1\in\cO$. By \eqref{eq:unit 1}
	with $\sigma=\id$ and $\eta$ is the nontrivial $K$-automorphism of $K(r)$,
	we have that $a_1\in\cO^*$. This proves Theorem~\ref{thm:EG} and we may even take
	$\{s_1,\ldots,s_N\}=\{r\}$. 
	
	\textbf{Case 2:} $[K(r):K]>2$. 
	Let $G_{r}$ be the subgroup of $L^{*}$
	generated by all the groups $\cO_{\sigma,\tau}^{*}$ and elements
	of the form $\sigma(r)-\eta(r)$ for any two distinct $K$-embeddings
	$\sigma$ and $\tau$ of $K(r)$ into $L$. By the definition of $d=d(\cO,r)$,
	 the rank of $G_r$ is bounded
	by a constant $c_9(d)$.
	
	For any
	two \emph{distinct nontrivial} $K$-embeddings
	$\sigma$ and $\eta$, Siegel's identity:
	$$\frac{q-\sigma(q)}{\eta(q)-\sigma(q)}-\frac{q-\eta(q)}{\eta(q)-\sigma(q)}=1$$
	gives that:
	$$(x_{q,\sigma,\eta},y_{q,\sigma,\eta}):=\left(\frac{q-\sigma(q)}{\eta(q)-\sigma(q)},
	-\frac{q-\eta(q)}{\eta(q)-\sigma(q)}\right)$$
	is a solution of the unit equation $x+y=1$
	to be solved for $(x,y)\in G_r^2$. Hence by
	Corollary~\ref{cor:2 variables}, there is a 
	finite
	set $S_r\subseteq L^{*}$
	whose cardinality is bounded above by a constant $c_{10}(d)$
	such that for every $q$ satisfying $\cO[q]=\cO[r]$
	and any two distinct nontrivial $K$-embeddings $\sigma,\eta$ of $K(r)$
	into $L$, we have:
	$$\left\{\frac{q-\sigma(q)}{\eta(q)-\sigma(q)},
    \frac{q-\eta(q)}{\eta(q)-\sigma(q)},\frac{q-\sigma(q)}{q-\eta(q)}\right\}\subseteq S_r.$$
	Now for any two pairs of distinct embeddings $(\sigma_1,\eta_1)$ and $(\sigma_2,\eta_2)$,
	using:
	$$\frac{\sigma_1(q)-\eta_1(q)}{\sigma_2(q)-\eta_2(q)}=\frac{(\sigma_1(q)-q)+(q-\eta_1(q))}
	{(\sigma_2(q)-q)+(q-\eta_2(q))}$$
	we conclude that there are only finitely many possibilities for $\displaystyle\frac{\sigma_1(q)-\eta_1(q)}{\sigma_2(q)-\eta_2(q)}$.
	
  Let $d'=[K(r):K]$ and $\bP:=\bP^{d'(d'-1)-1}$ with coordinates $x_{(\sigma,\tau)}$
	indexed by pairs $(\sigma,\tau)$ of distinct $K$-embeddings
	of $K(r)$ into $L$. We conclude that there is a finite set $T_r\subseteq \bP(L)$
 whose cardinality is bounded above by a constant $c_{11}(d)$
 such that for every algebraic integer $q$ satisfying $\cO_K[q]=\cO_K[r]$, the
 corresponding point $((\sigma(q)-\eta(q)))_{(\sigma,\eta)}$ 
	belongs to $T_r$. 
	
	Now if there are more than $c_{11}(d)$ many $q$ such that $\cO[q]=\cO[r]$, then there
	are at least two denoted by $q$ and $q^*$ such that the two points
	$((\sigma(q)-\eta(q)))_{(\sigma,\eta)}$ and $((\sigma(q^*)-\eta(q^*))_{(\sigma,\eta)}$
	in $\bP(L)$ coincide. In other words, there exists $u\in L^*$ such that:
	\begin{equation}\label{eq:q and q*}
  \frac{\sigma(q)-\eta(q)}{\sigma(q^*)-\eta(q^*)}=u\ \text{for any distinct $\sigma$ and $\eta$.}
	\end{equation}
	By \eqref{eq:unit 1}, we have that $u\in\cO_{\sigma,\eta}^*$ for any
	distinct $\sigma$ and $\eta$. 
	Since $K(r)=K(q)=K(q^*)$, by lifting to $\Gal(L/K)$ we have:
	$\displaystyle\frac{\tilde{\sigma}(q)-\tilde{\eta}(q)}{\tilde{\sigma}(q^*)-\tilde{\eta}(q^*)}=u$ 
	for every $\tilde{\sigma},\tilde{\eta}\in \Gal(L/K)$
	such that $\tilde{\sigma}\Gal(L/K(r))\neq\tilde{\eta}\Gal(L/K(r))$.	This implies $u$
	is invariant under $\Gal(L/K)$. Hence $u\in \cO^*$ thanks to integral closedness
	of $\cO$.
	Now \eqref{eq:q and q*} with $\sigma=\id$ implies that 
	the element $q-uq^*\in K(r)$
	 is invariant under every $K$-embedding of $K(r)$. Hence $q-uq^*\in\cO$. This proves
	 the first assertion in Theorem~\ref{thm:EG}.
	 
	For the remaining assertion, note that $\cO\subset \bar{\Q}^*$
	and $K$ is now a number field.  By Theorem~\ref{thm:effective} the finite 
	sets $S_r\subseteq L^*$ and $T_r\subseteq \bP(L)$
	can be determined effectively. Now we \emph{fix} a point $(t_{(\sigma,\eta)})\in T_r$ and 
	show how to effectively determine
	all algebraic integers $q$ such that $\cO[q]=\cO[r]$
	and the two points $(\sigma(q)-\eta(q))$
	and
	$(t_{(\sigma,\eta)})$ in $\bP(L)$ coincide. In other words, we need to determine
	$x\in L^*$ such that the system of equations: 
	\begin{equation}\label{eq:x}
	\sigma(q)-\eta(q)=t_{(\sigma,\eta)}x\ \text{for any $K$-embeddings $\sigma\neq\eta$
	of $K(r)$}
	\end{equation}
	could possibly yield a solution $q$ satisfying $\cO[q]=\cO[r]$. 
	
	Note that
	if $x\in L^*$ is a choice such that \eqref{eq:x}
	has a solution $q$ satisfying $\cO[q]=\cO[r]$
	then for every unit $w\in\cO^*$,
	$xw$ is another choice with a solution $qw$
	of \eqref{eq:x} satisfying $\cO_K[qw]=\cO_K[r]$. 
	Hence it suffices to determine
	the images of all such $x$ inside the quotien $L^*/\cO^*$.
	Write $t=\prod_{(\sigma,\eta)} t_{\sigma,\eta}$.
	The system
	\eqref{eq:x}
	together with Proposition~\ref{prop:discriminant}
	gives:
	\begin{equation}\label{eq:compare discriminant}
	x^{d'(d'-1)}t \in \disc_K(r)\cO^*.
	\end{equation}
	Denote $(\cO^*)^{d'(d'-1)}:=\{w^{d'(d'-1)}:\ w\in\cO^*\}$. Let
	$u_1,\ldots,u_M\in\cO_K^*$ be a choice of representatives for
	$\cO^*/(\cO^*)^{d'(d'-1)}$. To make such a choice, we simply
	need the group of roots of unity in $K$ and
	a choice of generators for the ``free part'' of $\cO^*$. This
	can be done effectively (compare Remark~\ref{rem:email}). Now \eqref{eq:compare discriminant}
	implies that 
	$$x\in \left(\frac{u_i\disc_K(r)}{t}\right)^{1/(d'(d'-1))} \cO^*$$
	for some $1\leq i\leq M$. Hence the list
	of possibilities for the image of $x$ in $L^*/\cO^*$
	can be effectively determined.
	
	To finish the proof, given $x\in L^*$, we explain
	how to find all solutions $q$ of
	\eqref{eq:x} satisfying $\cO[q]=\cO[r]$. Write
	$$q=a_0+a_1r+\ldots+a_{d'-1}r^{d'-1}$$
	and we solve for $(a_0,\ldots,a_{d'-1})$ in the free $\cO$-module $\cO^{d'}$
	instead. Write $t_{\eta}=t_{\sigma,\eta}$ if
	$\sigma$ is the identity. Restrict \eqref{eq:x} to the smaller system:
	\begin{equation}\label{eq:a_i}
	q-\eta(q)=t_{\eta}x\ \text{for any nontrivial $K$-embedding $\eta$ of $K(r)$}
	\end{equation}
	Then we have a linear system of $(d'-1)$ equations in
	the variables $a_1,\ldots,a_{d'-1}$. The rows of the coefficient matrix $C$
	are of the form
	$$(r-\eta(r),r^2-\eta(r^2),\ldots,r^{d'-1}-\eta(r^{d'-1}))$$
	where $\eta$ ranges over all nontrivial $K$-embeddings of $K(r)$. It is easy to see that
	$C$ is invertible, as follows. Let $D$ be the $d'\times d'$ Vandermonde matrix
	whose rows are of the form:
	$$(1,\eta(r),\ldots,\eta(r)^{d'-1})$$
	where $\eta$ ranges over all (including the identity) $K$-embeddings of $K(r)$. In
	particular, the first row of $D$ is $(1,r,r^2,\ldots,r^{d'-1})$. By applying elementary
	column operations to transform the top row to $(1,0,\ldots,0)$, we have that:
	$$\det(D)=\pm \det(C).$$
	Hence $\det(C)\neq 0$. Therefore there is a unique solution
	$(a_1,\ldots,a_{d'-1})\in \C^{d'-1}$. Now it depends on whether
	this unique solution $(a_1,\ldots,a_{d'-1})$ belongs to $\cO^{d'-1}$
	and whether
	$$q'=a_1r+\ldots+a_{d'-1}r^{d'-1}$$
	satisfies $\cO[q']=\cO[r]$.
	If that is the case, any $q=a_0+q'$ for any 
	$a_0\in\cO$
	satisfies $\cO[q]=\cO[r]$. Otherwise, our initial
	choices of $(t_{(\sigma,\tau)})$ and $x$
	do not yield any $q$ satisfying $\cO[q]=\cO[r]$.
	Verifying the condition $\cO[q']=\cO[r]$ could be
	done effectively by, for example, checking if the change of coordinate matrices
	between $\{1,q',\ldots,q'^{d'-1}\}$
	and $\{1,r,\ldots,r^{d'-1}\}$
	is in $\GL_{d'}(\cO)$.
	This finishes the proof of
	Theorem~\ref{thm:EG}.

	\begin{remark}\label{rem:email}
	In fact, the sets $S_r$ and $T_r$ in the proof
	above can be determined effectively thanks to
	the results of Evertse and Gy\H{o}ry \cite{EGunit2013} mentioned in
	Remark~\ref{rem:EGunit}. However, we are grateful to Professor Evertse for the
	explanation that results in \cite{EGunit2013} are not enough
	to effectively determine a list $\{s_1,\ldots,s_N\}$   
	in Theorem~\ref{thm:EG} for an arbitrary integrally closed finitely generated 
	domain $\cO$. The problem is that in the above proof, we work with
	$L^*/\cO^*$, hence require a list of generators of $\cO^*$. When
	$\cO\subset\bar{\Q}$ (i.e. $\cO$ is the ring of $S$-integers in a number field), 
	generators of $\cO^*$ can be determined effectively. However, this is not known
	for a general $\cO$ (see \cite[pp.~353]{EGunit2013}). 
	\end{remark}	
	
	\section{Proof of Theorem~\ref{thm:main 2}} 
	\subsection{Notation and some preliminary results}\label{subsec:notation}
	Throughout this section, fix an integrally closed finitely generated domain
	$\cO$ with fraction field $K$. Fix
	$r$ and $q$ that are integral over $\cO$ and satisfy $\{r^n,q^n:\ n\in\N\}\cap \cO=\emptyset$. Let $L$ denote
	the Galois closure of $K(q,r)$. 
	Write $d=d(\cO,q,r)$ defined in Definition~\ref{def:dAab}, and let $c_{12}(d),c_{13}(d),\ldots$ denote positive constants depending only 
	on $d$. Define $Q\in\N$ (respectively $R\in\N$) to be the smallest 
	positive 
	integer satisfying
	$K(q^Q)\subseteq K(q^n)$ (respectively $K(r^R)\subseteq K(r^n)$) for every $n\in\N$. We have:
	 \begin{lemma}\label{lem:Q and R}
	 The exists a bound $c_{12}(d)$ depending only on $d$ for $Q$ and $R$.
	 \end{lemma}
	 \begin{proof}
	 In fact, $Q$ are $R$ are bounded above by the
	 order $s$ of the group of roots of unity in $L^*$, as follows. For every $\sigma\in\Gal(L/K)$,
	if $\sigma(q^Q)=q^Q$ then $\sigma(q)/q$ is a root of unity. Hence we have 
	$\sigma(q^{Q-s})=q^{Q-s}$. This implies $\Gal(L/K(q^Q))\subseteq \Gal(L/K(q^{Q-s}))$,
	and hence $K(q^{Q-s})\subseteq K(q^Q)$ violating the minimality of $Q$ if $Q>s$. The same
	argument also shows $R\leq s$. Finally $s$ could be bounded explicitly in terms of the number of roots of unity in $K$ and $[L:K]$, hence
	in $d$.
		\end{proof}
	
	\begin{remark}\label{rem:Q and R}
	Note that the proof of Lemma~\ref{lem:Q and R} does \emph{not} use the 
	condition $\{q^n,r^n:\ n\in\N\}\cap \cO=\emptyset$.
	\end{remark}
	
	We need the following result for the proof of Theorem~\ref{thm:main 2}; it 
	is a special case of Corollary~\ref{cor:1.1}.
	\begin{proposition}\label{prop:special case}
	There is a constant $c_{13}(d)$ such that for every $n_0\in\N$, there are 
	at most $c_{13}(d)$ many $m\in\N$ (respectively $n\in\N$) such that $\cO[q^m]=\cO[q^{n_0}]$ (respectively $\cO[r^n]=\cO[r^{n_0}]$).
	\end{proposition}
	\begin{proof}
	It suffices to prove the assertion involving the identity 
	$\cO[q^m]=\cO[q^{n_0}]$ since the other assertion
	involving $\cO[r^n]=\cO[r^{n_0}]$ is completely analogous. 
	We use the same idea as in the proof of Theorem~\ref{thm:EG}. Since
	$q^n\notin K$ for every $n\in\N$, there is
	$\sigma\in \Gal(L/K)$
	such that $\sigma$ does not fix $q^n$ for every $n\in\N$. Suppose
	$\cO[q^m]=\cO[q^{n_0}]$, then there is a unit $u_{m}$
	of the ring $\cO[q,\sigma(q)]$
	such that 
	$$q^m-\sigma(q^m)=u_m(q^{n_0}-\sigma(q^{n_0})).$$
	Let $G$ be the subgroup of $L^*$ generated by the units in $\cO[q,\sigma(q)]$, $q$ and
	$\sigma(q)$. Then the rank of $G$ is bounded in terms of $d$ only. We have that
	$(u_m^{-1}q^m,-u_m^{-1}\sigma(q^m))\in G^2$
	is a solution of the equation:
	$$\frac{1}{q^{n_0}-\sigma(q^{n_0})}(x+y)=1.$$
	By Corollary~\ref{cor:2 variables},
	there are at most $c_{13}(d)$ possibilities
	for $\displaystyle\frac{q^m}{\sigma(q^m)}$.
	
	Hence, if there are more than $c_{13}(d)$ many $m$ such that
	$\cO[q^m]=\cO[q^{n_0}]$ then there are $m_1<m_2$
	such that 
	$$\frac{q^{m_1}}{\sigma(q^{m_1})}=\frac{q^{m_2}}{\sigma(q)^{m_2}}.$$
	In other words, $\sigma$ fixes $q^{m_2-m_1}$, contradicting the
	choice of $\sigma$.
	\end{proof}
	
	Finally, we have the following which can be proved using a similar idea:
	\begin{proposition}\label{prop:same field}
		There is a constant $c_{14}(d)$ such that
		if $K(r^R)\neq K(q^Q)$
		then there are at most $c_{14}(d)$ pairs
		$(m,n)\in\N^2$ such that 
		$\cO[q^m]=\cO[r^n]$.
	\end{proposition}
	\begin{proof}
		We may assume $K(q^Q)\nsubseteq K(r^R)$, hence $\Gal(L/K(r^R))$
		is not a subgroup of $\Gal(L/K(q^Q))$. Thus we can choose
		$\sigma\in \Gal(L/K(r^R))$ such that $\sigma$ does not fix
		$q^n$ for any $n\in\N$. Now it suffices to prove that there is 
		a constant
		$c_{15}(d)$ such that for every \emph{fixed}
		$0\leq \ell\leq R-1$
		there are at most $c_{15}(d)$ finitely many
		pairs $(m,n)$ satisfying $\cO[q^m]=\cO[r^n]$
		\emph{and} $n\equiv \ell$ modulo $R$. Once this is done, the
		desired $c_{14}(d)$ can be taken to be $Rc_{15}(d)$ (note Lemma~\ref{lem:Q and R}).
		
		For every such $(m,n)$, write $n=\tilde{n}R+\ell$.
		As before, there is a unit $u_{m,n}$ of the ring
		$\cO[q^m,\sigma(q^m),r^n,\sigma(r^n)]\subseteq \cO[q,\sigma(q),r,\sigma(r)]$ such that:
		$$q^m-\sigma(q^m)=u_{m,n}(r^n-\sigma(r^n))=u_{m,n}r^{\tilde{n}R}(r^{\ell}-\sigma(r^{\ell})).$$
	
	  Let $G$ be the subgroup of $L^*$ generated by the units of the ring
	  $\cO[q,\sigma(q),r,\sigma(r)]$, $q$,
	$\sigma(q)$, and $r$. Then the rank of $G$ is bounded in terms of $d$ only. 
	We have that
	$\displaystyle\left(\frac{q^m}{u_{m,n}r^{\tilde{n}R}},-\frac{\sigma(q^m)}{u_{m,n}r^{\tilde{n}R}}\right)\in G^2$
	is a solution of the equation:
	$$\frac{1}{r^{\ell}-\sigma(r^{\ell})}(x+y)=1.$$
	By Corollary~\ref{cor:2 variables},
	there is a constant $c_{16}(d)$ such that there are at most $c_{16}(d)$ possibilities
	for $\displaystyle\frac{q^m}{\sigma(q^m)}$. 
	
	Recall the constant $c_{13}(d)$ in Proposition~\ref{prop:special case}, define
	$c_{15}(d):=c_{13}(d)c_{16}(d)$. Now if there are more than $c_{15}(d)$ pairs
	$(m,n)$ with $\cO[q^m]=\cO[r^n]$ and $n\equiv \ell$ modulo $R$,
	then Proposition~\ref{prop:special case}
	implies that those pairs yield $N>c_{16}(d)$ many pairs
	denoted $(m_1,n_1),\ldots,(m_N,n_N)$
	such that $m_1,\ldots,m_N$ are distinct. We may assume $m_1<\ldots<m_N$. By
	the property of $c_{16}(d)$ as the upper bound for the possibilities
	of $\displaystyle\frac{q^m}{\sigma(q^m)}$, there exist
	$1\leq i<j\leq N$ such that:
	$$\frac{q^{m_i}}{\sigma(q^{m_i})}=\frac{q^{m_j}}{\sigma(q^{m_j})}.$$
	In other words, $\sigma$ fixes $q^{m_j-m_i}$, contradicting the choice of 
	$\sigma$. This finishes the proof.
	\end{proof}

	
	\subsection{Proof of Theorem~\ref{thm:main 2}}
	By Proposition~\ref{prop:same field}, we may assume that
	$K(q^Q)=K(r^R)$; denote this field by $K^o$.
	As in the proof of
	Proposition~\ref{prop:same field}, it suffices to fix $k,\ell$ with  
	$0\leq k\leq Q-1$
	and $0\leq \ell\leq R-1$, and show that there are at most $c_{17}(d)$
	pairs $(m,n)\notin\cA_{\cO,q,r}\cup\cB_{\cO,q,r}\cup\cC_{\cO,q,r}$ satisfying $\cO[q^m]=\cO[r^n]$, 
	$m\equiv k$ modulo $Q$, and $n\equiv \ell$ modulo $R$. Once this is done,
	the desired constant $c_3(d)$ in the conclusion of Theorem~\ref{thm:main 2}
	can be taken to be $QRc_{17}(d)$ (note Lemma~\ref{lem:Q and R}). The convenience of doing this is that we 
	can fix
	$F:=K(q^k)=K(q^m)=K(r^n)=K(r^{\ell})$. We have the following tower of fields:
	$$K\subsetneq K^o\subseteq F\subseteq L.$$
	
	Define: 
	$$W_{k,\ell}:=\{(m,n)\in \N^2:\ \cO[q^m]=\cO[r^n],\ m\equiv k\bmod Q,\ \text{and } n\equiv \ell\bmod R\}.$$ Let $G$ be the subgroup of $L^*$
	generated by the units of the rings $\cO[q,\sigma(q),r,\sigma(r)]$
	for all $\sigma\in \Gal(L/K)$ and by  all
	the conjugates of $q$ and $r$ over $K$.
	As in the proof of Theorem~\ref{thm:EG}, for every $(m,n)\in W_{k,\ell}$ and every $\sigma\in \Gal(L/K)\setminus \Gal(L/F)$ there is a unit $u_{m,n,\sigma}$
	of the ring 
	$\cO[q,\sigma(q^m),r,\sigma(r^n)]\subseteq \cO[q,\sigma(q),r,\sigma(r)]$ such that
	$$0\neq q^m-\sigma(q^m)=u_{m,n,\sigma}(r^n-\sigma(r^n)).$$
	
	Therefore $\displaystyle\bfx_{m,n,\sigma}:=
	\left(\frac{q^m}{\sigma(q^m)},-\frac{u_{m,n,\sigma}r^n}{\sigma(q^m)},
	\frac{u_{m,n,\sigma}\sigma(r^n)}{\sigma(q^m)}\right)$ is a solution of
	the unit equation
	\begin{equation}\label{eq:3 variables}
	x+y+z=1\ \text{with $(x,y,z)\in G^3$.} 
	\end{equation}
	Note that $\bfx_{m,n,\sigma}=\bfx_{m,n,\tau}$
	and $u_{m,n,\sigma}=u_{m,n,\tau}$
	if the two cosets $\sigma \Gal(L/F)$ and $\tau\Gal(L/F)$ coincide.
	Since the rank of $G$ is bounded in terms of $d$ only, by 
	Corollary~\ref{cor:3 variables}, 
	the number of nondegenerate solutions 
	is at most $c_{18}(d)$. We have the
	following:
	\begin{lemma}\label{lem:c18}
	Let $\sigma\Gal(L/F)$ be a coset 
	in $\Gal(L/K)$
	with $\sigma\notin\Gal(L/K^o)$,
	there are at most $c_{18}(d)$ many $(m,n)\in W_{k,\ell}$
	such that the solution $\bfx_{m,n,\sigma}$ of
	\eqref{eq:3 variables} is nondegenerate.
	\end{lemma}
	\begin{proof}
	Assume there are more than $c_{18}(d)$ pairs $(m,n)\in W_{\ell}$
	such that $\bfx_{m,n,\sigma}$ is degenerate. 
	Recall that $c_{18}(d)$ is a bound on
	the number of solutions of \eqref{eq:3 variables}, hence there are 
	two distinct pairs $(m_1,n_1)$ and $(m_2,n_2)$ 
	such that
	\begin{equation}\label{eq:c18}
	\bfx_{m_1,n_1,\sigma}=\bfx_{m_2,n_2,\sigma}.
	\end{equation}
	 We may assume $m_1\neq m_2$, the case $n_1\neq n_2$ is completely analogous.
	 Without loss of generality, assume $m_1<m_2$.
	 Equation \eqref{eq:c18} implies:
	 $$\frac{q^{m_1}}{\sigma(q^{m_1})}=\frac{q^{m_2}}{\sigma(q^{m_2})}.$$
	 In other words, $\sigma$ fixes $q^{m_2-m_1}$. Note that $m_2\equiv m_1\equiv k$ modulo $Q$. Hence the field $K(q^{m_2-m_1})\subseteq K^o$
	 is fixed by $\sigma$ and any element of $\Gal(L/K^o)$. Since 
	 $\sigma\notin\Gal(L/K^o)$, the field $K(q^{m_2-m_1})$ is strictly
	 smaller than $K^o=K(q^Q)$, contradicting the choice of $Q$.
	\end{proof}

	There are precisely $[F:K]-[F:K^o]< d$ cosets $\sigma\Gal(L/F)$
	in $\Gal(L/K)$ with $\sigma\notin \Gal(L/K^o)$. We define $c_{19}(d):=dc_{18}(d)$. We now complete the
	proof of Theorem~\ref{thm:main 2} by showing that
	there are at most 
	$c_{19}(d)$ pairs $(m,n)$ in
	$W_{k,\ell}\setminus (\cA_{\cO,q,r}\cup\cB_{\cO,q,r}\cup\cC_{\cO,q,r})$. Assume there
	are more than $c_{19}(d)$ such pairs. By Lemma~\ref{lem:c18}, there exists
	a pair $(\tilde{m},\tilde{n})\in W_{k,\ell}\setminus(\cA_{\cO,q,r}\cup\cB_{\cO,q,r}\cup\cC_{\cO,q,r})$
	such that the solution
	$\bfx_{\tilde{m},\tilde{n},\sigma}$
	of \eqref{eq:3 variables}
	is degenerate for \emph{every} coset $\sigma\Gal(L/F)$
	in $\Gal(L/K)$ with $\sigma\notin\Gal(L/K^o)$. We will show that this
	is impossible.
	
	For any coset $\sigma\Gal(L/F)$ with $\sigma\notin\Gal(L/K^o)$, degeneracy of
	$\bfx_{\tm,\tn,\sigma}$ falls into one of the following
	two types (note that
	we always have $\sigma(q^{\tm})\neq q^{\tm}$ since $\sigma\notin\Gal(L/F)$):
	\begin{itemize}
		\item [Type I: ] $q^{\tm}=u_{\tm,\tn,\sigma}r^{\tn}$ and 
		$u_{\tm,\tn,\sigma}\sigma(r^{\tn})=\sigma(q^{\tm})$. This implies
		that $\sigma$, hence every element in the coset $\sigma\Gal(L/F)$, fixes $\displaystyle \frac{q^{\tm}}{r^{\tn}}$.
		
		\item [Type II: ] $q^{\tm}=-u_{\tm,\tn,\sigma}\sigma(r^{\tn})$
		and $-u_{\tm,\tn,\sigma}r^{\tn}=\sigma(q^{\tm})$. This implies
		that $\sigma$, hence every element in the coset $\sigma\Gal(L/F)$, fixes $\displaystyle q^{\tm}r^{\tn}$.
	\end{itemize}
	 Note that it is possible that both types happen for the same coset $\sigma\Gal(L/F)$. Let $H_1:=\Gal\left(L/K\left(\displaystyle\frac{q^{\tm}}{r^{\tn}}\right)\right)$
	 and $H_2:=\Gal(L/K(q^{\tm}r^{\tn}))$. We have proved the following:
	 \begin{equation}\label{eq:union of 3}
	 \Gal(L/K)=\Gal(L/K^o)\cup H_1\cup H_2.
	 \end{equation} 
	 We need the following well-known lemma in group theory:
	 \begin{lemma}\label{lem:3subgroups}
	 \begin{itemize}
	 \item [(a)] A group cannot be the union of two proper subgroups. 
	 \item [(b)] If a group $A$ is the union of three proper subgroups $A_1$, $A_2$, and $A_3$
	 then $[A:A_i]=2$ for $1\leq i\leq 3$, $A_1\cap A_2=A_1\cap A_3=A_2\cap A_3$
	 is a normal subgroup of $A$ and the quotient is isomorphic to the Klein 
	 four-group.
	 \end{itemize} 
	 \end{lemma}
		\begin{proof}
			Part (a) is an easy exercise. Part (b) is a classical result attributed to 
			Scorza. See, for example,
			\cite{3subgroups} for a proof. 
		\end{proof}
		
		Then we have:
		\begin{lemma}\label{lem:H1 or H2}
		$\Gal(L/K)$ is equal to $H_1$ or $H_2$, but not both.
		\end{lemma}
		\begin{proof}
			Assume both $H_1$ and $H_2$ are proper subgroups of
			$\Gal(L/K)$ (note that $\Gal(L/K^o)$ is a proper
		subgroup by the assumption on $q$ and $r$), then Lemma~\ref{lem:3subgroups}
		implies that $\Gal(L/K^o)$, $H_1$, and $H_2$
		are distinct subgroups of index 2 in $\Gal(L/K)$. Hence the fields
		$K^o$, $K\left(\displaystyle\frac{q^{\tm}}{r^{\tn}}\right)$,
		$K(q^{\tm}r^{\tn})$ are distinct fields of degree 2 over $K$. Since
		the elements
		$\left(\displaystyle\frac{q^{\tm}}{r^{\tn}}\right)^{QR}$
		and $(q^{\tm}r^{\tn})^{QR}$ belong to the intersection of $K^o$
		with each of the remaining 2 fields, they belong to $K$.
		Hence $q^{2\tm QR}$ is in $K$, contradiction. Hence $\Gal(L/K)=H_1$
		or $\Gal(L/K)=H_2$.
		
		  The last assertion is easy: suppose $\Gal(L/K)=H_1=H_2$, then both
		  $\displaystyle\left(\frac{q^{\tm}}{r^{\tn}}\right)$ and $q^{\tm}r^{\tn}$
		  belong to $K$. Hence $q^{2\tm}\in K$, contradiction.
		\end{proof}
	
	We can now complete the proof of Theorem~\ref{thm:main 2} by showing that
	$(\tm,\tn)$ must belong to $\cA_{\cO,q,r}\cup\cB_{\cO,q,r}\cup\cC_{\cO,q,r}$. We divide
	into 2 cases:
	
	\textbf{Case 1:} $\Gal(L/K)=H_1$. This gives $\displaystyle\frac{q^{\tm}}{r^{\tn}}\in K$.
	We claim that there must be at least one coset $\sigma\Gal(L/F)$ with
	$\sigma\notin \Gal(L/K^o)$ such that the degeneracy of
	$\bfx_{\tm,\tn,\sigma}$ falls into Type I. Otherwise if
	every degeneracy is of Type II, we have $\Gal(L/K)=\Gal(L/K^o)\cup 
	H_2$ and Lemma~\ref{lem:3subgroups} gives that $\Gal(L/K)=H_2$ contradicting 
	Lemma~\ref{lem:H1 or H2}. Pick such a coset $\sigma\Gal(L/F)$ as claimed,
	then $\displaystyle\frac{q^{\tm}}{r^{\tn}}=u_{\tilde{m},\tilde{n},\sigma}$
	is a unit in $\cO[q,\sigma(q),r,\sigma(r)]$. Hence $\displaystyle\frac{q^{\tm}}{r^{\tn}}\in \cO^*$; in other words
	$(\tm,\tn)\in\cA_{\cO,q,r}$.
	
	\textbf{Case 2:} $\Gal(L/K)=H_2$. This gives $\alpha:=q^{\tm}r^{\tn}\in \cO$.
	By arguing as in Case 1, we can choose a coset $\eta\Gal(L/F)$ such that
	$\eta\notin\Gal(L/K^o)$
	and the degeneracy of $\bfx_{\tm,\tn,\eta}$ falls into Type II. Denoting $d'=[F:K]$,
	we can uniquely write:
	$$q^{\tm}=a_0+a_1r^{\tn}+\ldots+a_{d'-1}r^{\tn(d'-1)}$$
	for $a_0,\ldots,a_{d'-1}\in\cO$. Therefore:
	$$\alpha=q^{\tm}r^{\tn}=a_0r^{\tn}+\ldots+a_{d'-1}r^{\tn d'}.$$
	Since $d'=[K(r^{\tn}):K]$, we must have $a_{d'-1}\neq 0$
  and
  $$X^{d'} + \frac{a_{d'-2}}{a_{d'-1}}X^{d'-1}+\ldots+
  \frac{a_0}{a_{d'-1}}X-\frac{\alpha}{a_{d'-1}}$$
	is the minimal polynomial of $r^{\tn}$ over $K$. In particular:
	\begin{equation}\label{eq:norm 1}
		\frac{\alpha}{a_{d'-1}}=\pm \Norm_{F/K}(r^{\tn}).
	\end{equation}
	where $\Norm_{F/K}$ denotes the norm map associated to the extension $F/K$. This implies:
	\begin{equation}\label{eq:norm 2}
	q^{\tm}=\frac{\alpha}{r^{\tn}}=\frac{\pm\Norm_{F/K}(r^{\tn})a_{d'-1}}{r^{\tn}}.
	\end{equation}
	Our choice of the coset $\eta\Gal(L/F)$ gives:
	\begin{equation}\label{eq:eta}
		q^{\tm}=-u_{\tm,\tn,\eta}\eta(r^{\tn})
	\end{equation}
	Together with \eqref{eq:norm 2}, we have:
	\begin{equation}\label{eq:eta vs norm}
		\pm \Norm_{F/K}(r^{\tn})a_{d'-1}=u_{\tm,\tn,\eta}r^{\tn}\eta(r^{\tn}).
	\end{equation}
	We now have two subcases:
	
	\textbf{Case 2.1:} $d'=[F:K]=2$, then $\eta(r^{\tn})$ is the
	conjugate of $r^{\tn}$ that is different from $r^{\tn}$. Equation \eqref{eq:eta vs norm}
	gives that $a_{d'-1}=\pm u_{\tm,\tn,\eta}$ is a unit
	in the ring $\cO[q,\eta(q),r,\eta(r)]$. Since $a_{d'-1}\in\cO$, we
	have that $a_{d'-1}\in\cO^*$. Finally, equation \eqref{eq:norm 2} gives that
	$q^{\tm}=u\eta(r^{\tn})$ for a unit $u\in \cO^*$. This means $(\tm,\tn)\in\cB_{\cO,q,r}$.
	
	\textbf{Case 2.2:} $d'=[F:K]>2$. Equation \eqref{eq:eta vs norm} implies that 
	$a_{d'-1}\in\cO^*$ and some conjugate of $r^{\tn}$ is a unit over $\cO$
	(see Definition~\ref{def:unit over}), hence $r^{\tn}$ itself is also
	a unit over $\cO$. Finally equation \eqref{eq:norm 2} gives that 
	$q^{\tm}r^{\tn}\in\cO^*$. This means $(\tm,\tn)\in \cC_{\cO,q,r}$.
	
	In conclusion, we have the contradiction that $(\tm,\tn)\in\cA_{\cO,q,r}\cup\cB_{\cO,q,r}\cup\cC_{\cO,q,r}$. Hence the set
	$$W_{k,\ell}\setminus(\cA_{\cO,q,r}\cup\cB_{\cO,q,r}\cup\cC_{\cO,q,r})$$
	must have at most $c_{19}(d)$ elements. This gives at most $QRc_{19}(d)$ pairs 
	$(m,n)\in\N^2$
	outside $\cA_{\cO,q,r}\cup\cB_{\cO,q,r}\cC_{\cO,q,r}$ satisfying
	$\cO[q^m]=\cO[r^n]$. Lemma~\ref{lem:Q and R} finishes the proof
	of Theorem~\ref{thm:main 2}.
	
	\section{An addendum to Theorem~\ref{thm:main 2}}\label{sec:add}
	For the sake of completeness, we explain how to describe
	all pairs $(m,n)$ such that
	$\cO[q^m]=\cO[r^n]$ without the condition that
	$\{q^n,r^n:\ n\in\N\}\cap\cO=\emptyset$. This section, though
	somewhat lengthy due to various cases to be considered, is rather elementary. 
	The notations $d$,
	$L$, $Q$, and $R$ are as in the beginning of Section~\ref{subsec:notation}. 
	Without loss of generality, we assume that $r^R\in\cO$. We consider two cases.
	
	\subsection{The case $q^Q\notin \cO$}
	For every $0\leq \ell\leq R-1$, define:
	$$W_{\ell}:=\{(m,n)\in\N^2:\ \cO[q^m]=\cO[r^n]\ \text{and } n\equiv \ell\bmod R\},$$
	$$V_{\ell}:=\{m\in\N:\ (m,n)\in W_{\ell}\ \text{for some $n$}\}=\pi_1(W_{\ell}),$$
	where $\pi_1$ is the projection from $\N^2$ onto its first factor.
	
	Since $q^m\notin K$ for every $m\in\N$, we have that 
	$W_{0}=\emptyset$. Hence
	it suffices to consider $\ell>0$. We have the following:
	\begin{proposition}\label{prop:first case}
	\begin{itemize}
	\item [(a)] There is a constant $c_{20}(d)$ such that 
	for every $0<\ell\leq R-1$, the set
	$V_{\ell}$ has at most $c_{20}(d)$ elements. 
	\item [(b)] Given $\ell$ and
	$m$
	with $0<\ell\leq R-1$ and $m\in V_{\ell}$. Then either the set
	$$U_{m}:= \{(m,n)\in W_{\ell}\}$$
	is a singleton or $r^R\in \cO^*$. 
	\item [(c)] If $r^R\in \cO^*$ then for every
	$0<\ell\leq R-1$ and every $m\in V_{\ell}$, we have
	$$U_{m}=\{(m,\ell+jR):\ j\in \N\cup\{0\}\}.$$
	\end{itemize}
	\end{proposition}
	\begin{proof}
	For part (a), we use the same arguments as in the proof of Proposition~\ref{prop:same field}. Pick
	$\sigma\in\Gal(L/K)$ such that $\sigma$ does not fix $q^m$ for any $m\in\N$. For
	every $(m,n)\in W_{\ell}$, write $n=\tilde{n}R+\ell$. As in the proof
	of Proposition~\ref{prop:same field}, we have:
	$$q^m-\sigma(q^m)=u_{m,n,\sigma}r^{\tilde{n}R}(r^{\ell}-\sigma(r^{\ell}))$$
	which gives at most $c_{20}(d)$ possibilities for
	$\displaystyle\frac{q^m}{\sigma(q^m)}$. Hence there are at most $c_{20}$
	possibilities for $m$ because of the choice of $\sigma$.
	
	For part (b), we note that $U_{m}\neq \emptyset$ since
	$m\in V_{\ell}$. Assume there are $n_1<n_2$ such that
	$(m,n_1),(m,n_2)\in W_{\ell}$. As before, write
	$n_i=\tn_iR+\ell$ for $i=1,2$. From the equation
	$$q^{m}-\sigma(q^m)=u_{m,n_1,\sigma}r^{{\tn_1}R}(r^{\ell}-\sigma(r^{\ell}))=u_{m,n_2,\sigma}r^{\tn_2R}(r^{\ell}-\sigma(r^{\ell})),$$
	we have that $\displaystyle r^{(\tn_2-\tn_1)R}$ and, hence,
	$r^R$ belong to $\cO^*$ by integral closedness of $\cO$.

	Part (c) is immediate since we have $\cO[r^n]=\cO[r^{\ell}]$
	if $n\equiv \ell\bmod R$.
	\end{proof}

	
	Proposition~\ref{prop:first case} finishes the case $q^Q\notin \cO$.
	
	\subsection{The case $q^Q\in\cO$}  As in the proof of Theorem~\ref{thm:main 2},
	we
	fix $0\leq k\leq Q-1$ and $0\leq \ell\leq R-1$ and describe the set:
	$$W_{k,\ell}:=\{(m,n)\in \N^2:\ \cO[q^m]=\cO[r^n],\ m\equiv k\bmod Q\ \text{and } n\equiv \ell\bmod R\}.$$
	It is easy to show that $W_{k,0}=W_{0,\ell}=\emptyset$ if $k\neq 0$ and $\ell\neq 0$. On
	the other hand:
	$$W_{0,0}=Q\N\times R\N.$$
	Hence from now on we may assume $k,\ell>0$. We also assume $K(q^k)=K(r^{\ell})$
	and denote this field by $F$ (otherwise $W_{k,\ell}=\emptyset$). We have
	the tower of fields:
	$$K\subsetneq F\subseteq L.$$
	
	As before, for every $(m,n)\in W_{k,\ell}$ and every $\sigma\in\Gal(L/K)$ with $\sigma\notin\Gal(L/F)$, there is a
	unit $u_{m,n,\sigma}$ (depending only on the coset $\sigma\Gal(L/F)$)
  such that:
	$$q^m-\sigma(q^m)=u_{m,n,\sigma}(r^n-\sigma(r^n)).$$
	Note that $\sigma$ fixes $q^{m-k}$ and $r^{n-\ell}$, we have:
	\begin{equation}\label{eq:m-k n-l}
	\frac{q^{m-k}}{r^{n-\ell}}=u_{m,n,\sigma}\frac{r^{\ell}-\sigma(r^{\ell})}
	{q^k-\sigma(q^k)}
	\end{equation}
	which depends only on $k$, $\ell$, and the coset $\sigma\Gal(L/F)$. If
	there are two distinct pairs $(m_1,n_1)$ and $(m_2,n_2)$ in $W_{k,\ell}$,
	equation \eqref{eq:m-k n-l} gives:
	\begin{equation}\label{eq:m1m2n1n2}
	\frac{q^{m_2-m_1}}{r^{n_2-n_1}}=\frac{u_{m_2,n_2,\sigma}}{u_{m_1,n_1,\sigma}}\in\cO^*
	\end{equation}
	since it is a unit over $\cO$ (see Definition~\ref{def:unit over}) and belongs to $K$ due to $Q\mid m_1-m_2$ and
	$R\mid n_2-n_1$.
	
	Note that if $r$ is a unit over $\cO$ then $r^R\in\cO^*$, hence $\cO[r^n]=\cO[r^{\ell}]$
	for every $n\equiv \ell\bmod R$. Moreover, if $q$ is not
	a unit over $\cO$ and $(m_1,n_1),(m_2,n_2)\in W_{k,\ell}$
	then \eqref{eq:m1m2n1n2} gives that $m_1=m_2$.
	Hence we have the following:
	\begin{proposition}\label{prop:non-unit}
		\begin{itemize}
		\item [(a)] If both $q$ and $r$ are units over $\cO$ then
		$W_{k,\ell}$ is either empty or has the form
		$$(k,\ell)+Q\N\times R\N.$$
		\item [(b)] If $r$ is a unit over $\cO$ and $q$ is not, then either $W_{k,\ell}$
		is empty
		or has the form
		$$\{(m_1,n):\ n\equiv \ell\bmod R\}$$
		where $m_1$ is the only positive integer such that $\cO[q^{m_1}]=\cO[r^{\ell}].$ A completely analogous statement holds
		when $q$ is a unit over $\cO$ and $r$ is not.
		\item [(c)] Assume that neither $q$ nor $r$ is a unit over $\cO$. If $|W_{k,\ell}|\geq 2$
		then the following holds. There is a minimal pair $(M,N)\in\N^2$
		satisfying $\displaystyle\frac{q^{QM}}{r^{RN}}\in\cO^*$. For any 2 distinct
		pairs
		$(m_1,n_1),(m_2,n_2)\in W_{k,\ell}$, we have $(m_2-m_1)(n_2-n_1)>0$.
		Moreover, we have $\displaystyle\frac{m_2-m_1}{QM}=\frac{n_2-n_1}{RN}$ and it is an integer.
		\end{itemize}
		\begin{proof}
		Perhaps only part (c) needs further explanation. 
		The assertion $(m_2-m_1)(n_2-n_1)>0$ follows from \eqref{eq:m1m2n1n2}
		and the assumption that $q$ and $r$ are not unit over $\cO$.
		The set of $(M,N)\in\Z^2$
		such that $\displaystyle{q^{QM}}{r^{RN}}\in \cO^*$  is a $\Z$-module
		of rank at most 1 since neither $q$ nor $r$ is a unit over $\cO$.
		
		In fact, this $\Z$-module has rank 1 and a basis $(M,N)\in\N^2$ due to \eqref{eq:m1m2n1n2}. As a consequence, for any
		distinct $(m_1,n_1),(m_2,n_2)\in W_{k,\ell}$
		the pair $\displaystyle\left(\frac{m_2-m_1}{Q},\frac{n_2-n_1}{R}\right)$
		is an integral multiple of the basis $(M,N)$. This implies
		the last assertion of (c).
		\end{proof}
	\end{proposition}
	  
	\begin{assume}\label{assume:valuations}
	The following assumption is \emph{only needed when one is concerned with the effectiveness
	of the results in the rest of this section}. Since $\cO$ is a Noetherian 
	integrally closed domain,
	if $q$ is not a unit (respectively $r$ is not a unit), there are only
	finitely many minimal primes
	ideal $\fq$ (respectively $\fr$) containing $q$ (respectively $r$), 
	each of the $\fq$ (respectively $\fr$) has height 1, and the localization
	$\cO_{\fq}$ (respectively $\cO_{\fr}$)
	is a DVR \cite{Matsumura}.
	For each such $\fq$ (respectively $\fr$), let
	$v_{\fq}$ denote the corresponding valuation on $K$ 
	normalized so that $v_{\fq}(K^*)=\Z$
	(respectively $v_{\fr}(K^*)=\Z$).
	\emph{We make the following assumption: it is possible 
	to
	effectively determine all the minimal primes $\fq$ (respectively $\fr$)
	containing $q$ (respectively $r$) and 
	to compute an extension on $L$ for each
	of the valuations $v_{\fq}$ (respectively $v_{\fr}$).} 
	\end{assume}
	
	\begin{remark}\label{rem:effective in bc}
		Under Assumption~\ref{assume:valuations},
		by choosing \emph{one} minimal prime $\fq$ containing $q$ and using
		the valuation $v_{\fq}$, 
		we can effectively determine the number $m_1$ in 
		\eqref{eq:m-k n-l}. The pair $(M,N)$ in part (c) can 
		be determined effectively by using \emph{all}
		the valuations $v_{\fq}$ and $v_{\fr}$. If such pair $(M,N)$ does not exist,
		then either $W_{k,\ell}=\emptyset$
		or $W_{k,\ell}$ contains at most one element;
		in both cases the set $W_{k,\ell}$ can be computed thanks to \eqref{eq:m-k n-l}.
	\end{remark}
	  
	\emph{From now on, we assume that neither $q$ nor $r$ is a unit, there
	is a minimal pair $(M,N)\in\N^2$ satisfying 
	$\displaystyle\frac{q^{QM}}{r^{RN}}\in\cO^*$, and $W_{k,\ell}\neq \emptyset$}. Part (c) of Proposition~\ref{prop:non-unit}
shows that $W_{k,\ell}$ has the minimal element denoted 
$(\tm,\tn)$ such that every $(m,n)$ in $W_{k,\ell}$
has the form $(\tm+tQM,\tn+tRN)$
for some $t\in \N\cup\{0\}$. We finish this section by solving the following two
problems:
 \begin{itemize}
 	\item [I:] explain how to obtain an upper bound  for $\tm$ and $\tn$. Once
 	this is done, by verifying the equation $\cO[q^m]=\cO[r^n]$
 	for $m$ and $n$ within such a bound, we could decide whether $W_{k,\ell}$
 	is empty or not.
 	
 	\item [II:] given $(\tm,\tn)$, explain how to find all $t$ such that
 	$(\tm+tQM,\tn+tRN)\in W_{k,\ell}$. This completely describes $W_{k,\ell}$.
 \end{itemize} 	 
	For the rest of this section, $c_{21},c_{22},\ldots$ denote 
	positive constants
	depending on $K$, $q$, and $r$. These constants can be computed
	under Assumption~\ref{assume:valuations}. We have:
	\begin{lemma}\label{lem:C16}
	There is a positive constant $c_{21}\geq 1$ 
	such that:
	$\tn \leq c_{21}\tm$ and $\tm\leq c_{21}\tn$.
	\end{lemma}
	\begin{proof}
	By picking one valuation $v_{\fq}$ and one valuation $v_{\fr}$, we can
	use \eqref{eq:m-k n-l} to prove this lemma easily.
	\end{proof}

	Let $d':=[F:K]\geq 2$,
	we can uniquely write:
	\begin{align}\label{eq:a b}
		\begin{split}
	q^{\tilde{m}}&=a_0+a_1r^{\tn}+a_2 r^{2\tn}+\ldots+a_{d'-1}r^{(d'-1)\tn}\\
	r^{\tilde{n}}&=b_0+b_1q^{\tm}+b_2 q^{2\tm}+\ldots+b_{d'-1}q^{(d'-1)\tm}
		\end{split}
	\end{align}
	for $a_0,\ldots,a_{d'-1},b_0,\ldots,b_{d'-1}\in\cO$. Denote $u=\displaystyle\frac{q^{QM}}{r^{RN}}\in\cO^*$.
	For every $t\in\Z$, using $q^{tQM}=u^tr^{tRN}$ and $r^{tRN}=u^{-t}q^{tQM}$, we have the following:
	\begin{align}\label{eq:alpha beta}
	\begin{split}
		q^{\tm+tQM}&=\alpha_0+\alpha_1 r^{\tn+tRN}+\alpha_2 r^{2(\tn+tRN)}+\ldots
		+\alpha_{d'-1}r^{(d'-1)(\tn+tRN)}\\
		r^{\tn+tRN}&=\beta_0+\beta_1 q^{\tn+tQM}+\beta_2 q^{2(\tm+tQM)}+\ldots
		+\beta_{d'-1}q^{(d'-1)(\tm+tQM)}
	\end{split}
	\end{align}
	where $\displaystyle\alpha_i=\frac{a_iu^t}{r^{(i-1)tRN}}\in K$
	and  $\displaystyle\beta_i=\frac{b_iu^{-t}}{q^{(i-1)tQM}}\in K$ for
	$0\leq i\leq d'-1$. Now we can give an upper bound for $\tm$ and $\tn$:
	\begin{proposition}\label{prop:C17}
	Define $c_{22}=c_{21}\max\{QM,RN\}$. We have $\max\{\tm,\tn\}\leq c_{22}$.
	\end{proposition}
	\begin{proof}
		Assume otherwise: $\max\{\tm,\tn\}> c_{22}$. Then Lemma~\ref{lem:C16}
		gives that $\tm> QM$ \emph{and} $\tn> RN$. By
		\eqref{eq:a b} and the fact that $\displaystyle\frac{q^{QM}}{r^{RN}}\in\cO^*$,
		we have that $a_0\in r^{RN}\cO$ and $b_0\in q^{QM}\cO$. 
		By \eqref{eq:alpha beta} when $t=-1$, we have that
		$\alpha_i,\beta_i\in\cO$ for $0\leq i\leq d'-1$; it suffices to check
		this when $i=0$ only. In other words, we have
		$\cO[q^{\tm-QM}]=\cO[r^{\tn-RN}]$
		violating the minimality of $(\tm,\tn)$.		
	\end{proof}
	
	Since we have bounded $(\tm,\tn)$ in terms of $K$, $q$, and $r$, the bounds given below,
	which apparently depend on $(\tm,\tn)$, indeed depend only on $K$, $q$, and $r$. 
	It is obvious that the two conditions ``$a_i\neq 0$ for some $2\leq i\leq d'-1$''
	and ``$b_j\neq 0$ for some $2\leq j\leq d'-1$'' are equivalent since $q^{\tm}$
	is linear in $r^{\tn}$ iff $r^{\tn}$ is linear in $q^{\tm}$. When $a_i=b_j=0$
	for every $2\leq i,j\leq d'-1$, the equation $\cO[q^{\tm}]=\cO[r^{\tn}]$ 
	is equivalent to $a_1$ or $b_1$, hence both, are units.
	The following result concludes this section:
	\begin{proposition}
	If $d'\geq 3$ and $a_i\neq 0$ for some $2\leq i\leq d'-1$ (hence $b_j\neq 0$ for some $2\leq j\leq d'-1$), there exists a positive constant $c_{23}$ such that 
	every $(m,n)\in W_{k,\ell}$ satisfies $\displaystyle \frac{m-\tm}{QM}=\frac{n-\tn}{RN}\leq c_{23}$. Consequently, $W_{k,\ell}$ has
	at most $c_{23}+1$ elements.
	
	On the other hand, if $a_i=b_i=0$ for every $2\leq i\leq d'-1$ (this is
	vacuously true when $d'=2$), then:
	$$W_{k,\ell}=\{(\tm+tQM,\tn+tRN):\ t\in\N\}.$$
	\end{proposition}
	\begin{proof}
	To prove the first assertion, pick $2\leq i\leq d'-1$ such that $a_i\neq 0$.   
	By using a valuation $v_{\fr}$,
	we have that 
	for sufficiently large $t\in\N$, the
	coefficient:
	$$\alpha_i=\frac{a_iu^t}{r^{(i-1)tRN}}$$
	cannot belong to $\cO$, hence $q^{\tm+tQM}\notin\cO[r^{\tn+tRN}]$.
	
	For the second assertion, for every $t\in\N$, we have $\alpha_0,\alpha_1,\beta_0,\beta_1\in\cO$
	while $\alpha_i=\beta_i=0$ for $2\leq i\leq d'-1$. This gives
	 $\cO[q^{\tm+tQM]}=\cO[r^{\tn+tRN}]$.
	\end{proof}
	  
	\section{Final remarks and further questions}\label{sec:last}
	\subsection{Effectiveness of our results}
	The effectiveness of Theorem~\ref{thm:EG} and results in 
	Section~\ref{sec:add} has been discussed.  
	For Theorem~\ref{thm:main 2}
	we note that it is \emph{not} effective in the sense that 
	we cannot provide a bound for the pairs $(m,n)\notin\cA_{\cO,q,r}\cup\cB_{\cO,q,r}\cup\cC_{\cO,q,r}$ satisfying
	$\cO[q^m]=\cO[r^n]$. The reason is that the theorem of Evertse, Schlickewei and Schmidt is not effective. Its proof relies
	crucially on a quantitative absolute version of the Subspace
	Theorem by Evertse and Schlickewei \cite{ESquantitative} after
	seminal work by Schmidt \cite{Schmidt72}, \cite{Schmidt89}. The question
	of making the Subspace Theorem effective is still wide open. We ask
	the following:
	\begin{question}
	Let $\cO$, $q$, $r$, $\cA_{\cO,q,r}$, $\cB_{\cO,q,r}$, $\cC_{\cO,q,r}$ be as 
	in 
	Theorem~\ref{thm:main 2}. Provide a bound (depending on $\cO$, $q$, and $r$) 
	for all 
	pairs $(m,n)\in\N^2$
	such that $\cO[q^m]=\cO[r^n]$
	and $(m,n)\notin \cA_{\cO,q,r}\cup\cB_{\cO,q,r}\cup\cC_{\cO,q,r}$.
	\end{question}
	
	\subsection{Another result by Bell and Hare}
	Besides the results mentioned previously in this paper, Bell and Hare \cite[Theorem~1.5]{BellHare} prove the following:
	\begin{theorem}[Bell-Hare]\label{thm:40}
	Let $r$ be an algebraic integer of degree at most 3. Then there are at most
	40 Pisot numbers $q$ such that $\Z[q]=\Z[r]$.
	\end{theorem}
	They cannot find an example of $r$ (of degree 3) that gives more than 7 Pisot numbers $q$ satisfying $\Z[q]=\Z[r]$, and ask for an improvement to
	the bound 40. Unfortunately, the bound in Corollary~\ref{cor:2 questions}
	when $\cO=\Z$ and $d=3$ is \emph{much larger} than 40. The proof of
	Theorem~\ref{thm:40} uses results on cubic Thue equations $F(x,y)=1$
	by Bennett \cite[Theorem~1.4]{Bennett2001}. 
	
	\subsection{Another approach to Theorem~\ref{thm:main 2}} 
	In \cite[Theorem~1.1]{BellHare}, 
	Bell and Hare actually study the
	equation $\disc_{\Q}(q^n)=\disc_{\Q}(r^n)$. Using Definition~\ref{def:discriminant}, they expand both sides to
	conclude that a certain linear recurrence sequence vanishes at $n$. Their
	definition of being ``full rank'' mentioned at the beginning of this paper 
	makes it relatively easy to study the degeneracy of the resulting linear recurrence sequence.
	
	On the other hand, we can ask the problem of describing all $(m,n)$
	such that $\displaystyle\frac{\disc_K(q^m)}{\disc_K(r^n)}\in\cO^*$. Again,
	we can use Definition~\ref{def:discriminant} to 
	expand $\disc_K(q^m)$ and $\disc_K(r^n)$ and get a unit equation, then 
	Theorem~\ref{thm:ESS} provides
	a bound on the number of nondegenerate solutions. However, there are two 
	issues. First, it does not seem entirely obvious how to get the exact relation
	(such as the relations described in the sets $\cA_{\cO,q,r}$, $\cB_{\cO,q,r}$,
	and $\cC_{\cO,q,r}$) from ``too many'' degenerate solutions. Second, by studying
  the property 
  $\displaystyle\frac{\disc_K(q^m)}{\disc_K(r^n)}\in\cO^*$ alone, we can never 
  rule out the case, say,  
	$q^m=u\sigma(r^n)$ for some conjugate $\sigma(r^n)$ of $r^n$ and some $u\in\cO^*$.
	On the other hand, Theorem~\ref{thm:main 2} indicates that 
	(except finitely many $(m,n)$) the case $q^m=u\sigma(r^n)$ (with $\sigma(r^n)\neq r^n$) can only happen when
	$q^m$ and $r^n$ have degree 2 over $K$.

	\bibliographystyle{amsalpha}
	\bibliography{Pisot-Revision1} 	

\end{document}